\documentclass[12pt]{amsart}

\usepackage{graphics, amsmath, amsfonts, amssymb, amsthm, amscd, mathrsfs}

\input{xy}
\xyoption{all}

\newtheorem{theorem}{Theorem}

\newtheorem{lemma}[theorem]{Lemma}

\theoremstyle{definition}
\newtheorem{definition}[theorem]{Definition}
\newtheorem{example}[theorem]{Example}
\newtheorem{remark}[theorem]{Remark}

\newcommand{\PP}{{\mathbb P}}
\newcommand{\g}{{\mathfrak g}}

\newcommand{\C}{{\mathbb C}}
\newcommand{\N}{{\mathbb N}}
\renewcommand{\O}{{\mathscr O}}
\newcommand{\M}{{\mathscr M}}
\newcommand{\R}{{\mathbb R}}
\newcommand{\Z}{{\mathbb Z}}
\newcommand{\Aut}{{\operatorname{Aut}}}

\newcommand{\VFA}{ \operatorname{{\rm VF}_{alg}}}
\newcommand{\LieA}{ \operatorname{{\rm Lie}_{alg}}}

\def\SAut{\mathop{\rm SAut}}

\newcommand{\G}{\ensuremath{\mathbb{G}}}
\def\reg{{\mathop{\rm reg}}}

\title{Flexibility properties in Complex Analysis and Affine Algebraic Geometry}

\author{Frank Kutzschebauch}

\address{Frank Kutzschebauch, Institute of Mathematics, University of Bern, Sidlerstrasse 5, CH-3012 Bern, Switzerland}
\email{frank.kutzschebauch@math.unibe.ch}


\date
\today
\thanks{F.~Kutzschebauch was partially supported by Schweizerischer Nationalfond grant 200021-140235/1.  }

\begin{document}\begin{abstract}  

In the last decades  affine algebraic varieties and Stein manifolds with big (infinite-dimensional) automorphism groups have been intensively studied. Several notions expressing that the automorphisms group is big have been proposed.
All of them imply that the manifold in question is an Oka-Forstneri\v c manifold. This important notion has also recently merged from  the intensive studies around the homotopy principle in Complex Analysis. This homotopy principle, which goes back to the 1930's, has had an enormous impact
on the development of the area  of Several Complex Variables and the number of its applications is constantly growing. In this overview article we present 3 classes of properties: 1. density property, 2. flexibility 3. Oka-Forstneri\v c. For each class we give the relevant definitions, its
 most significant features and explain the known implications between all these properties. Many difficult mathematical problems could be solved by  applying the developed theory, we indicate some of the most spectacular ones.
\end{abstract}

\maketitle
\tableofcontents

\section{Introduction}  \label{s:introduction}
This  is a survey of recent developments in Complex Analysis and Affine Algebraic Geometry which emphasize on objects  described as elliptic, in the opposite to hyperbolic in the sense of Kobayashi or more general
in the sense of Eisenman (all Eisenman measures on these objects vanish identically.)

Here is the scheme of properties we are going to discuss, together with the known implications between them. Although the properties do not require the manifolds to be Stein or affine algebraic, some of the implications do. We therefore  assume the manifold to be Stein in the upper row and to be a smooth affine algebraic variety in the lower row.
\bigskip

$$\begin{matrix}\label{scheme}

 \rm{density\:  property\:  (DP)}                          &             \Longrightarrow                            &   \rm{holomorphic\:  flexible}                       &                   \Longrightarrow                     &      {\rm Oka-Forstneric}   \cr
\Uparrow                                         &         &                           \Uparrow               &     &   \cr
 \rm{algebraic\:  density\: property\:  (ADP)}          &                                        &    \rm{algebraic\: flexible}                    &                                      &           \cr

\end{matrix} \eqno{(1)}$$

\bigskip
In each of the following three sections we present one class of  properties  together with main features.  
We hope that researchers from the algebraic and from the holomorphic side can join their efforts to enlarge the class of examples and to find out which of the reverse implications between these properties hold.

In the last section we briefly recall that in the presence of a volume form there is a similar property to (algebraic) density property, called (algebraic) volume density property, which if replacing DP and ADP in scheme $(1)$
by these properties (AVDP, VDP) gives another scheme with the same implications true. Also we elaborate  on the reverse implications in our scheme $(1)$.

We sincerely thank the referee for very carefully reading the article and making many valuable comments. He helped a lot to improve  the presentation. Many thanks to Finnur L\'arusson for careful reading and catching an inaccuracy in a previous version of the text.

%
%
%

\section{Density Property}

\subsection{Definition and main features}
Considering a question promoted by Walter Rudin, Anders\'en and Lempert in 1989 \cite{AL} proved a remarkable fact about the affine $n$-space $n\ge 2$, namely that the group generated by shears (maps of the form $(z_1, \ldots, z_n) \mapsto (z_1, \ldots , z_{n-1}, z_n + f (z_1, \ldots , z_{n-1}))$ where $f \in \O (\C^{n-1})$ is a holomorphic function and any linear conjugate of such a map) and overshears (maps of the form $(z_1, \ldots, z_n) \mapsto (z_1, \ldots , z_{n-1}, z_n  g (z_1, \ldots , z_{n-1}))$ where $g \in \O^\ast (\C^{n-1})$ is a nowhere vanishing holomorphic function and any linear conjugate of such a map) are dense in holomorphic automorphism group of $\C^n$, endowed with compact-open topology. The main importance of their work was not the mentioned   result but the proof itself which implies, as observed by Forstneri\v c and Rosay in \cite{FR} for $ X= \C^n$, the  remarkable  Anders\'en-Lempert theorem, see below. The natural generalization from $\C^n$ to arbitrary manifolds $X$ was made by Varolin \cite{V1} who introduced the following important
property of a complex manifold:

\begin{definition}\label{1.10}
A complex manifold $X$ has the density property if in the
compact-open topology the Lie algebra  generated by
completely integrable holomorphic vector fields on $X$ is dense in
the Lie algebra  of all holomorphic vector fields on
$X$.
 \end{definition}

Here a holomorphic vector field $\Theta$ on a complex manifold $X$  is called completely integrable if the ODE

$$\frac{d}{dt} \varphi (x,t) = \Theta (\varphi (x,t))$$

$$\varphi (x,0) = x$$

\noindent
has a solution $\varphi (x,t)$ defined for all complex times $t \in \C$ and all starting points $x \in X$. It gives  a complex one-parameter subgroup in the holomorphic
automorphism group $\Aut_{hol} (X)$.

The density property is a precise way of saying that the automorphism group of a manifold is big, in particular for a Stein manifold this is underlined by the main result of the theory (see \cite{FR} for $\C^n$, \cite{V1},  a detailed proof can be found in the Appendix of \cite{R} or in  \cite{For}).

\begin{theorem}[\bf Anders\'en-Lempert theorem] \label{AL-Theorem} Let $X$ be a Stein manifold with the density  property
and let  $\Omega$ be an open subset of $X$. Suppose that  $ \Phi : [0,1] \times \Omega \to X$ is a $C^1$-smooth
map  such that

{\rm (1)} $\Phi_t : \Omega \to X$ is holomorphic and injective   for every  $ t\in [0,1]$,

{\rm (2)} $\Phi_0 : \Omega \to X$ is the natural embedding of $\Omega$ into $X$, and

{\rm (3)} $\Phi_t  (\Omega) $ is  a Runge  
subset\footnote{Recall that an open subset $U$ of $X$ is Runge if any holomorphic function on $U$ can
be approximated by global holomorphic functions on $X$ in the compact-open topology. Actually, for $X$ Stein by Cartan's Theorem A this definition
implies more: for any coherent sheaf on $X$ its section over $U$ can be approximated by global sections.}  of $X$ for every $t\in [0,1]$.

Then for each $\epsilon >0 $ and every compact subset $K \subset \Omega$ there is a continuous family,
$\alpha: [0, 1] \to \Aut_{hol} (X)$  of  holomorphic  automorphisms of $X$
such that  $$\alpha_0 = id \, \, \, {\rm and} \, \, \,\vert \alpha_t - \Phi_t \vert_K <\epsilon {\rm \ \  for\ \  every\ \ } t \in [0,1]$$ 

\end{theorem}

Philosophically one can think of the density property as a tool for realizing  local movements  by  global  maps   (automorphisms).    In some sense it is a substitute  for cutoff functions which in the differentiable category are used for globalizing local movements. In the holomorphic category we of course loose control on automorphism
outside the compact set $K$. This makes constructions more complicate but still constructing  sequences of automorphisms  by iterated use of the Anders\'en-Lempert theorem
has led to remarkable constructions.

Let us further remark that the implications of the density property for manifolds which are not Stein have not been explored very much yet. If the manifold is compact all (holomorphic) vector fields are completely integrable, the
density property trivially hold and thus cannot give any further information on the manifold.

\begin{remark} Anders\'en and Lempert proved that every algebraic vector field on $\C^n$ is a finite sum of  algebraic shear fields (fields of form $ p(z_1, \ldots z_{n-1}) \frac \partial {\partial z_n}$ for a polynomial $p \in \C[\C^{n-1}]$ and their linear conjugates, i.e.,  fields who's one-parameter subgroups consist of shears)  and overshear fields 
(fields of form $ p(z_1, \ldots z_{n-1}) z_n  \frac \partial {\partial z_n}$ for a polynomial $p \in \C[\C^{n-1}]$ and their linear conjugates, i.e.,  fields whose one-parameter subgroups consist of overshears). Together with the fact that any holomorphic automorphism of $\C^n$ can be joined to the identity by a smooth pat, this shows how the Anders\'en-Lempert theorem implies that the group generated by shears and overshears is dense in the holomorphic automorphism group of $\C^n$
\end{remark}

The algebraic density property can be viewed as a tool to prove the density property, whereas the ways of proving it are purely algebraic work.

\begin{definition}
An affine algebraic manifold $X$ has the algebraic density
property if the Lie algebra $\LieA (X)$ generated by completely
integrable algebraic vector fields on it coincides with  the Lie
algebra $\VFA (X)$ of all algebraic vector fields on it.
\end{definition}

An algebraic vector field is an algebraic section of the tangent bundle, for example on  $\C^n$ it can be written as $\sum_{i=1}^n p_i (z_1, \ldots , z_n) \frac{\partial}{\partial z_i}$ with polynomials $p_i \in \C[\C^n]$. If it is completely integrable, 
its flow gives a one parameter subgroup in the holomorphic automorphism group not necessarily in the algebraic automorphism group. For example,  a polynomial shear field of the form $p (z_1, \dots , z_{n-1}) z_n  \frac{\partial}{\partial z_n}$
has the flow map $\gamma (t, z) = (z_1, \ldots , z_{n-1}, \exp (t p (z_1, \dots , z_{n-1}) ) z_n)$. This is the reason that algebraic density property is in the intersection of affine algebraic geometry and complex analysis. It is an algebraic notion, proven using algebraic methods
but has implications for the holomorphic automorphism group.

\subsection{Applications and examples}

A first application we like to mention is to the notoriously difficult question whether every open Riemann surface can be properly holomorphically embedded into $\C^2$. This is the only dimension for which   the conjecture of Forster \cite{Fo}, saying
that every Stein manifold of dimension $n$ can be properly holomorphically embedded into $\C^N$ for $N= [\frac{n}{2}] + 1$, is still unsolved. The conjectured dimension is sharp by examples of Forster \cite{Fo} and has been proven by Eliashberg, Gromov \cite{EG} and Sch\"urmann \cite{Sch} for all dimensions $n\ge 2$.
Their methods of proof fail in dimension $n=1$. But Fornaess Wold invented a clever combination of a use of shears (nice  projection property) and Theorem \ref{AL-Theorem} which led to  many new embedding theorems for open Riemann surfaces. As an example
we like to mention the following two recent results of Forstneri\v c and Fornaess Wold \cite{FW}, \cite{FW1} the first of them being the most general one for open subsets of the complex line:

\begin{theorem} Every domain in the Riemann sphere with at least one  and at most countably many boundary components, none of which are points, admits a proper holomorphic embedding into $\C^2$.
\end{theorem}

\begin{theorem}  If $\bar \Sigma$ is a (possibly reducible) compact complex curve in $\C^2$ with boundary $\partial \Sigma$ of class $C^r$ for some $ r > 1$, then the inclusion map 
$i: \Sigma = \bar \Sigma\setminus \Sigma \to \C^2$ can be approximated, uniformly on compacts in $\Sigma$, by proper holomorphic embeddings $\Sigma \to \C^2$. 
\end{theorem}

Many versions of embeddings with interpolation are also known and proven using the same  methods invented by Fornaess Wold in \cite{W} .

Another application is to construct non-straightenable holomorphic embeddings of $\C^k$ into $\C^n$ for all pairs of dimensions $0<k<n$, a fact which is contrary to the situation in affine algebraic geometry, namely contrary to the famous Abhyankar-Moh-Suzuki theorem for $k=1, n=2$ and also to work of Kaliman \cite{Ka} or  $2k+1 < n$, whereas straightenability for the other dimension pairs is  still unknown in algebraic geometry. The most recent and quite striking  result in this direction says that there are even holomorphic families of pairwise non-equivalent  holomorphic embeddings (refering to holomorphic automorphisms of the source and target in the definition below).
Here non-straightenable for an embedding  $\C^k$ into $\C^n$ means to be not equivalent to the standard embedding.

\begin{definition}\label{def-eq-emb}
Two embeddings $\Phi,\Psi\colon X\hookrightarrow\C^n$ are {\it equivalent} if there 
exist automorphisms $\varphi\in\Aut(\C^n)$ and $\psi\in\Aut(X)$ such that 
$\varphi\circ\Phi=\Psi\circ\psi$.
\end{definition}

\begin{theorem}  see \cite{Kutzschebauch-Lodin}.
  \label{main1} Let $n, l$ be natural numbers with $n\ge l+2$.
There exist, for $k=n-l-1$,  a family of holomorphic embeddings of
$\C^l$ into $\C^n$ parametrized by $\C^k$, such that for
different parameters $w_1\neq w_2\in \C^k$ the embeddings
$\psi_{w_1},\psi_{w_2}:\C^l \hookrightarrow \C^{n}$ are
non-equivalent.
\end{theorem}

We would like to mention a nice application of Theorem \ref{main1} to actions of compact (or equivalently complex
reductive, see \cite{Ku}) groups on $\C^n$. It was a long standing problem,  whether all holomorphic actions of such groups on affine space are linear after a change of variables (see for example the overview article \cite{Huckleberry}). The first counterexamples to that (Holomorphic Linearization) problem were constructed by Derksen and the first author in \cite{DK1}. The method from \cite{DK1} is holomorphic in a parameter and therefore  applied to our parametrized situation leads to the following result (\cite{Kutzschebauch-Lodin})

\begin{theorem} \label{action1}
For any $n\ge 5$ there is a holomorphic family of $\C^*$-actions on $\C^n$ para\-metrized by $\C^{n-4}$
$$\C^{n-4} \times \C^* \times \C^n \to \C^n , \quad  (w, \theta, z) \mapsto \theta_w (z)$$ so that for
different parameters $w_1\neq w_2\in \C^{n-4}$ there is no equivariant isomorphism between the actions $\theta_{w_1}$ and  $\theta_{w_2}$.
\end{theorem}

The linearization problem for holomorphic $\C^*$-actions on $\C^n$ is thus solved to the positive for $n=2$ by Suzuki \cite{Su} and still open for $n=3$. For $n=4$ there are uncountably many actions (non-linearizable ones among them)
\cite{DK} and for $n\ge 5$ Theorem \ref{action1}  implies that there are families. Moreover, there are families including a linear action as a single member of the family;

\begin{theorem} \label{action2}
For any $n\ge 5$ there is a holomorphic family of $\C^*$-actions on $\C^n$ parametrized by $\C$
$$\C \times \C^* \times \C^n \to \C^n\quad (w, \theta, z) \mapsto \theta_w (z)$$ so that for
different parameters $w_1\neq w_2\in \C$ there is no equivariant isomorphism between the actions
$\theta_{w_1}$ and $\theta_{w_2}$. Moreover, the action $\theta_0$ is linear.
\end{theorem}

\smallskip\noindent
\textbf{Open Problem:}  Suppose $X$ is a Stein manifold with density property and $Y \subset X$ is a closed submanifold. Is there always another proper holomorphic embedding $\varphi : Y \hookrightarrow X$ which is not equivalent to
the inclusion $ i: Y \hookrightarrow X$?
\smallskip

We should remark that an affirmative answer  to this question is stated in \cite{V2}, but the author apparently had another (weaker) notion of equivalence in mind.

Here comes  the essentially complete  list of examples of Stein manifolds known to have the density property:

\smallskip\noindent
\textbf{List of examples of Stein manifolds known to have the density property:}

\smallskip
(1)  $X=G/R$ where $G$ is linear algebraic and $R$ a reductive subgroup has ADP and thus DP (defined on p.2), except for $X=\C$ and $X= (\C^*)^n$.
(this includes all examples known from the work of Anders\'en-Lempert and Varolin and Varolin-Toth and  Kaliman-Kutzschebauch, the final result is proven by Donzelli-Dvorsky-Kaliman \cite{DDK});

(2) The manifolds  $X$ given as a submanifold in $\C^{n+2}$ with coordinates $u\in \C$, $v\in \C$, $z\in \C^n$  by the equation $uv = p(z)$, where the zero fiber of the polynomial $p \in \C[\C^n]$ is smooth (otherwise $X$ is not smooth), have ADP
\cite{KaKu2}.

(3) The only known non-algebraic example with DP are the manifolds  $X$ given as a submanifold in $\C^{n+2}$ with coordinates $u\in \C$, $v\in \C$, $z\in \C^n$  by the equation $uv = f(z)$, where the zero fiber of the holomorphic function  $f \in \O (\C^{n})$ is smooth (otherwise $X$ is not smooth) \cite{KaKu2}. 

(4) Danilov-Gizatullin surfaces have ADP \cite{D}.

A variant of density property for (normal) singular varieties (considering vector fields vanishing on a  subvariety in particular  on the singular locus $Sing(X)$) was introduced in \cite{KLL}. A version of  Anders\'en-Lempert theorem holds in this situation which allows to approximate local movements taking place in the smooth part of $X$ by automorphisms fixing the singular locus.   It is proven in  \cite{KLL} that normal affine toric varieties have this property. Another version of this generalization considering
holomorphic automorphisms of $\C^n$ fixing  a codimension two subvariety can be found in \cite{KaKu1}. For more information on the density property we refer to the overview article \cite{KaKu3}.

\section{Flexibility}
\subsection{Definition and main features}

The notion of flexibility is the most recent among the described properties. It was defined in \cite{AFKKZ}. First the algebraic version: 

\begin{definition}
Let $X$ be a reduced  algebraic variety defined over $\C$ (any algebraically closed field would do). 
We let $\SAut(X)$ denote the subgroup of
$\Aut_{alg} (X)$ generated by all algebraic one-parameter unipotent
subgroups of $\Aut_{alg} (X) $, i.e., algebraic subgroups isomorphic to the
additive group $\G_a$ (usually denoted $\C^+$ in complex analysis). The group $\SAut(X)$ is called the
{\em special automorphism group} of $X$; this is a normal subgroup
of $\Aut_{alg} (X)$.
\end{definition}

\begin{definition}
 We say that a point $x\in X_\reg$ is
{\em algebraically   flexible} if the tangent space $T_x X$ is spanned by the
tangent vectors to the orbits $H.x$ of one-parameter unipotent
subgroups $H\subseteq\Aut_{alg} (X)$. A variety $X$ is called {\em algebraically  flexible} if
every point $x\in X_\reg$ is.\end{definition}
Clearly, $X$ is algebraically flexible if one
point of $X_\reg$ is and the group $\Aut_{alg} (X)$ acts transitively on
$X_{\rm reg}$.

The main feature of algebraic flexibility is the following result from \cite{AFKKZ} (whose proof mainly relies on the Rosenlicht theorem);

\begin{theorem}\label{mthm}
For an  irreducible affine variety $X$ of dimension $\ge 2$,
the following conditions are equivalent.
\begin{enumerate}
\item The group $\SAut (X)$ acts transitively on  $X_\reg$.
\item The group $\SAut (X)$ acts
infinitely transitively on $X_\reg$.
\item  $X$ is an algebraically  flexible variety.
\end{enumerate}
\end{theorem}

 The paper \cite{AFKKZ} also contains versions of simultaneous transitivity (where the space $X_\reg$ is stratified by orbits of  $\SAut (X)$) and versions with jet-interpolation. Moreover, it was recently remarked 
 that the theorem holds for quasi-affine varieties, see Theorem 1.11. in \cite{FKZ}.
 
The holomorphic version of this notion is much less explored, it is obviously implied by the algebraic version in case $X$ is an algebraic variety.

\begin{definition} We say that a point $x\in X_\reg$ is
{\em holomorphically   flexible} if the tangent space $T_x X$ is spanned by the
tangent vectors of completely integrable holomorphic vector fields, i.e. holomorphic  one-parameter 
subgroups in $\Aut_{hol} (X)$. 
A complex manifold  $X$ is called {\em holomorphically  flexible} if
every point $x\in X_\reg$ is.\end{definition}
Clearly, $X$ is holomorphically flexible if one
point of $X_\reg$ is and the group $\Aut_{hol} (X)$ acts transitively on
$X_{\rm reg}$.

In the holomorphic category it is still open whether an analogue  of Theorem \ref{mthm} holds.

\smallskip\noindent
\textbf{Open Problem:} Are the three equivalences from Theorem  \ref{mthm}  true for an irreducible   Stein  space $X$?  More precisely, if an irreducible  Stein  space $X$ is holomorphically flexible, does the holomorphic automorphism group 
$\Aut_{hol} (X)$ act infinitely  transitively on
$X_{\rm reg}$? 
\smallskip

It is clear that holomorphic flexibility of $X$ implies that $\Aut_{hol} (X)$ acts transitively on
$X_{\rm reg}$, i.e., the implication $ (3) \Rightarrow (1)$ is true.
Indeed, let $\theta_i, i=1, 2, \ldots, n$ be completely integrable holomorphic vector fields which span the tangent space $T_x X$ at some point $x \in X_{\rm reg}$, where $n= dim X$. If $\psi^i : \C \times X \to X, \quad (t, x)  \mapsto \psi^i_ t (x)$ denote the corresponding one-parameter subgroups, then
the map $\C^n \to X, \quad (t_1, t_2, \ldots, t_n) \mapsto \psi_{t_n}^n \circ \psi_{t_{n-1}}^{n-1} \circ \cdots \circ \psi_{t_1}^1 (x)$ is of full rank at $t =0$ and thus by the Inverse Function Theorem a local biholomorphisms from a neighborhood of $0$ to a neighborhood of $x$. Thus the $\Aut_{hol} (X)$-orbit
through any point of $X_{\rm reg}$ is open. If all orbits are open, each orbit is  also closed, being the complement of all other orbits. Since $X_{\rm reg}$ is connected, this implies that it consists of one orbit.

The inverse implication  $ (1) \Rightarrow (3)$ is also true. For the proof we appeal to the Hermann-Nagano Theorem which states that if $\g$ is a Lie algebra of holomorphic vector fields on a manifold
$X$, then  the orbit $R_\g  (x) $ (which is the union of all points $z$ over any collection of finitely many fields $v_1, \ldots v_N \in \g $  and over all times $(t_1, \ldots, t_N)$ for  which the expression $z =\psi_{t_N}^N \circ \psi_{t_{N-1}}^{N-1} \circ \cdots \circ \psi_{t_1}^1 (x)$ is defined)
  is a locally closed submanifold and its tangent space at any point $y \in R_\g  (x)$ is $T_y R_\g (x) = span_ { v \in \g} {v (y)}$. We consider the Lie algebra $\g$
generated by completely integrable holomorphic vector fields. Since by the assumption the orbit is $X_{\rm reg}$ we conclude that Lie combinations of completely integrable holomorphic vector fields span the tangent space at each point in  $X_{\rm reg}$. Now suppose at some point $x_0$ the completely integrable fields do not generate $T_{x_0} X_{\rm reg}$, i.e., there is a proper linear subspace $W$ of $T_{x_0} X_{\rm reg}$, such that $v(x_0) \in W$ for all completely integrable  holomorphic fields $v$. Any Lie combination of completely integrable  holomorphic fields is a limit (in the compact open topology) of sums of completely integrable  holomorphic fields due to the formula $\{ v, w\} = lim_{t \to 0} \frac {\phi_t^* (w) - w} t$ for the Lie bracket ($\phi_t^* (w)$ is a completely integrable field pulled back by an automorphism, thus completely integrable!). Therefore all Lie combinations of completely integrable fields evaluated  at  $x_0$ are contained in $W \subset T_{x_0} X_{\rm reg}$, a contradiction.

In order to prove the  remaining implication $ (3) \Rightarrow (2)$ one would like  to find suitable  functions $f \in {\rm Ker }  \theta$ for a completely integrable holomorphic vector field $\theta$, vanishing at one point and not vanishing at some other point of $X$.  In general these functions may not exist, an orbit of $\theta$ can be dense in $X$.

At this point it is worth mentioning that for a Stein manifold DP implies all three conditions from Theorem \ref{mthm}. For flexibility this is lemma \ref{span} below,  infinite transitivity
(with jet-interpolation) is proved by Varolin in \cite{V2}.
%
%
%
%
%

Also the generalized form of DP  for Stein spaces defined in \cite{KLL} implies all three conditions from Theorem \ref{mthm}. 

\subsection{Examples}

Examples of algebraically flexible varieties are homogeneous spaces of semisimple Lie groups (or extensions of semisimple Lie groups by unipotent radicals), toric varieties without non-constant invertible regular functions, cones over flag varieties and cones over Del Pezzo surfaces of degree at least $4$, normal hypersurfaces of the form
$uv = p(\bar x)$ in $\C_{u, v, \bar x}^{n+2}$. Moreover, algebraic subsets of codimension at least $2$ can be removed as recently shown by Flenner, Kaliman and Zaidenberg in \cite{FKZ}

\begin{theorem}  Let $ X$ be a smooth quasi-affine variety of dimension $\ge 2$ and $Y \subset  X$ a closed subscheme of codimension $\ge  2$. If $X$ is flexible then so is $X\setminus Y$.

\end{theorem}

\section{Oka-Forstneri\v c manifolds}

\subsection{Historical introduction to Oka theory and motivational examples}
The notion of Oka-Forstneri\v c manifolds is quite new (it was introduced by Forstneri\v c in \cite{For1}, who called them Oka manifolds following a suggestion  of L\' arusson who underlined the importance of such a notion already in  \cite{La}) but the development merging into this important notion, called Oka theory, has a long history.  It started with Oka's theorem from 1939 that the second (multiplicative) Cousin problem on a domain of holomorphy  is solvable with holomorphic functions if it is solvable with continuous functions. This  imples that a holomorphic line bundle on such a domain is is holomorphically trivial if it is topologically trivial.

Let us recall how the generalizations of the classical one variable results  of Mittag-Leffler (construct meromorphic functions with prescribed main parts) and Weierstrass (construct meromorphic functions with prescribed zeros and poles)
are generalized to several complex variables.

Let us recall the first (additive) Cousin problem  which arrises from the Mittag-Leffler problem, generalizing the famous Mittag-Leffler theorem from one variable to several variables: Given data  $ \{(U_i, m_i)\} $, where ${U_i}$ is an open cover of a complex space $X$ and $m_i \in \M (U_i)$ is a meromorphic function on $U_i$ such that 
every difference $f_{ij} = m_i\vert_{U_{ij}} - m_j\vert_{U_{ij}}$ is holomorphic on $U_{ij} = U_i \cap U_j$, find a global meromorphic function $m \in \M (X)$ on $X$ such that $m\vert_{U_i} - m_i $ is holomorphic on $U_i$ for all $i$.

For solving  this Mittag-Leffler problem  one first solves  the associated additive Cousin  problem, defined as follows: The collection $f_{ij} \in \O (U_{ij})$ defines a 1-cocycle on the cover ${U_i}$ with values in the sheaf $\O$ of holomorphic functions, meaning that for each triple $i, j, k$ of
indexes we have 
$$ f_{ij} + f_{jk}+ f_{ki} = 0 \quad {\rm on} \quad  U_{ijk} = U_i \cap U_j \cap U_k .$$

Given such a 1-cocycle $\{ f_{ij} \} $, the Cousin I problem asks for a collection of holomorphic functions $f_j \in \O (U_j)$ (a 0-cochain) such that 

$$ f_i - f_j = f_{ij} \quad {\rm on} \quad U_{ij}.$$ 

One expresses this by saying the cocycle splits or it is a 1-coboundary.
From the solution to the additive Cousin problem one obtains by setting  $m\vert_{U_i} = m_i - f_i$  a well-defined (since $m_i -m_j = f_{ij} = f_i -f_j \quad {\rm on} \quad U_{ij}$) global meromorphic function $m \in \M (X)$ solving the Mittag-Leffler problem.

The vanishing of the first Cech cohomology group $H^1 (X, \O)$  with coefficients in the sheaf $\O$ means that every 1-cocycle splits on a refinement of the covering. In other words  $H^1 (X, \O) = 0$ implies  that every 1-cocycle becomes a 1-coboundary on a refinement, so every Mittag-Leffler problem 
is solvable, in particular by Cartan's Theorem B this is true for any Stein manifold.

The second (multiplicative) Cousin Problem arises from the problem of finding meromorphic functions with prescribed zeros and poles, solved by Weierstrass in one variable. Given data  $ \{(U_i, m_i)\} $, where ${U_i}$ is an open cover of a complex space $X$ and $m_i \in \M^* (U_i)$ is an invertible (i.e., not vanishing identically on any connected component) meromorphic function on $U_i$ such that for any pair of indexes the quotient $f_{ij} := g_i g_j^{-1}$ is a nowhere vanishing holomorphic function $f_{ij} \in \O^* (U_{ij})$. Our data defines a divisor $D$ on $X$ and the problem is to find a 
global meromorphic function $m \in \M (X)$ defining this divisor, meaning, such a function  that $m m_i^{-1}$ is a nowhere vanishing holomorphic function on $U_i$  for every $i$. A solution is obtained by solving the second Cousin problem:  Given a collection ${f_{ij}}$ of nowhere vanishing
holomorphic functions $f_{ij } : U_{ij} \to \C^*$ satisfying the 1-cocycle condition

$$f_{ii} = 1 \quad f_{ij} f_{ji}  = 1 \quad f_{ij} f_{jk} f_{ki}  = 1$$

\noindent
on $U_i$, $U_{ij}$,  $U_{ijk}$ respectively, find nowhere vanishing holomorphic functions $f_j : U_j \to \C^*$ such that 

$$ f_i = f_{ij} f_j \quad \mathrm{ on} \ \ U_{ij} $$.

If such $f_i$ exist then $ g_i f_i^{-1}  = g_j f_j^{-1} $ on $U_{ij}$ which defines a solution, a meromorphic function  $m \in \M (X)$ representing our divisor.

The following cohomological  formulation and proof of Oka's Theorem are standard, see e.g.  \cite{For} Theorem 5.2.2..

\begin{theorem} If  X is a complex space satisfying $H^1 (X,\O) = 0$ then the homomorphism $H^1 (X,\O^*) \to H^1 (X,\C^*)$ induced by the sheaf inclusion $\O^* \hookrightarrow \C^*$ is injective. In particular
if a multiplicative Cousin problem is solvable by continuous functions, then it is solvable by holomorphic functions. If in addition we have $H^2 (X, \O) = 0 $ then the above map is an isomorphism.

\end{theorem}
\begin{proof}
Consider the exponential sheaf sequence (where $\sigma (f) = e^{2 \pi i f}$).

\[ \xymatrix{
0                                \ar[r] \ &   \Z  \ar[r] \ar[d]^{id} & \O \ar[r]^\sigma \ar[d]  &  \O^* \ar[r]  \ar[d] & 1 \\
0                                \ar[r] &\ Z                    \ar[r]                & \C  \ar[r]^\sigma          & \C^*  \ar[r]          & 1 } \]

Since due to partition of unity $H^1 (X, \C) = H^2 (X, \C) = 0$ the relevant portion of long exact cohomology sequence is:

\[ \xymatrix{
H^1 (X, \Z)      \ar[r]  &  H^1 (X \O)   \ar[r] \ar[d] & H^1 (X, \O^*) \ar[r]^{c_1} \ar[d]  &  H^2 (X, \Z)  \ar[r]  \ar@{=}[d] & H^2 (X, \O) \ar[d] \\
                              &  0                 \ar[r]          & H^1 (X, \C^*)  \ar[r]^{c_1}          &  H^2 (X, \Z)  \ar[r]          & 0 } \]

The map in the bottom row is an isomorphism  $H^1 (X, \C^*)  \cong H^2 (X, \Z)$. If $H^1  ( X, \O) = 0$ the (1-st Chern class) map $c_1$ in the first row is injective
$ 0 \to H^1 (X, \O^*) \xrightarrow{c_1}     H^2 (X, \Z) \cong H^1 (X, \C^*) $. If in addition  $H^2 (X, \O) = 0$ this map is an isomorphism.

\end{proof}

By Oka's theorem on a complex space with $H^1 (X, \O) = H^2 (X, \O) = 0$ (by Theorem B this holds in particular on a Stein space) the natural map from equivalence classes of holomorphic line bundles into equivalence classes of continuous (complex) line bundles is an isomorphism.
For higher rank vector bundles this   cohomological proof fails due to non commutativity  of the relevant cohomology groups. Nevertheless, Grauert was able to prove the corresponding statement in higher dimensions. The following theorem is the holomorphic counterpart of Quillen's and Suslin's result that projective modules over affine space are free.

\begin{theorem} For a Stein space $X$
the natural map $\rm Vect_{hol}^r (X) \to \rm Vect_{top}^r (X)$ of equivalence classes of rank $r$ complex vector bundles is a bijection for every $r \in \N$.
\end{theorem}

This theorem follows from the following result, named  Grauert's Oka-principle by H.Cartan, obtained by Grauert \cite{Grauert}, Grauert and Kerner \cite{GK} and Ramspott \cite{R} (see Theorem 5.3.2. in \cite{For}).

\begin{theorem} \label{GOP}
If $X$ is a Stein space and $\pi: Z \to X$ is a holomorphic fibre bundle with a complex homogeneous fibre whose structure group is a complex Lie group acting transitively on the fibre, then the inclusion 
$\Gamma_{hol} (X, Z) \hookrightarrow \Gamma_{cont} (X, Z)$ of the space of global holomorphic sections into the space of global continuous sections is a weak homotopy equivalence. In particular every continuous section
is homotopic to a holomorphic section. 
\end{theorem}

An equivariant version of Grauerts Oka principle with respect to an action of a reductive complex Lie group 
has been proven by Heinzner and Kutzschebauch \cite{HK}. This principle  in particular implies that the method of constructing counterexamples
to the linearization problem, found by Schwarz in the algebraic setting \cite{Schwarz}, does not work in the holomorphic category. Moreover, the above mentioned Oka principle was recently used by Kutzschebauch, L\' arusson and Schwarz  \cite{KLS} to show, among others a strong linearization result:  A generic holomorphic action, which is locally over a
common categorical quotient isomorphic to a linear action on $\C^n$, is in fact globally isomorphic to that
linear action.

The next step in Oka theory was made by Gromov in his seminal paper \cite{Gromov}, which marks  the beginning of modern Oka theory. He introduced the notion of dominating spray and ellipticity (see the last section). The great improvement compared to Grauert's Oka-principle is the fact that not the fibre together with the transition maps of the bundle, but only certain  properties of the fibre totally independent of transition maps allow to derive the conclusions. In the above cited classical works , the structure group was indeed assumed to be a complex Lie group. However, in modern Oka theory the structure group is completely irrelevant.  Moreover, modern Oka theory allows to consider sections of stratified elliptic submersions, generalizing the case of locally trivial fibre bundles. The emphasis shifted from the cohomological to the homotopy theoretic
aspect, focusing on those analytic properties of a complex manifold $Y$ which ensure that every continuous map from a Stein space $X$ to $Y$ is homotopic to a holomorphic map, with natural additions concerning approximation and interpolation of maps that are motivated by the extension and approximation theorems for holomorphic functions on Stein spaces. The approximation and extension is needed for generalizing from maps $X \to Y$ (which can be considered as  sections of the trivial bundle $X \times Y \to X$ with fibre $Y$)  to sections of holomorphic submersions $Z\to X$ with Oka-Forstneri\v c fibres and moreover,  to stratified elliptic submersions.
\subsection{Definition and main features}

\begin{definition} A complex manifold $Y$ is an Oka-Forstneri\v c manifold if every holomorphic map  $f: K \to Y $ from (a neighborhood of) a compact convex set  $K\subset \C^n$ (any dimension $n$) can be approximated uniformly on $K$ by entire maps $\C^n \to Y$.
\end{definition}
The property in this definition is also called Convex Approximation Property (CAP), if the dimension $n$ is fixed we speak of $\rm (CAP)_n$, thus (CAP) means $\rm (CAP)_n$ for all $n$.  By work of Forstneri\v c (CAP) is equivalent to any of thirteen different Oka properties, one of them is mentioned in the following Theorem which includes all all versions of the classical Oka-Grauert principle discussed  in the Introduction.  This theorem answers Gromov's question whether Runge approximation on a certain class of compact sets in Euclidean spaces suffices to infer the Oka property. Since all these thirteen Oka properties are equivalent characterizations of the same class of manifolds Forstneri\v c called them Oka manifolds. In order to honor his work on the equivalence of all the Oka properties the author finds the notation Oka-Forstneri\v c manifolds more appropriate.

\begin{theorem} 
 Let $ \pi: Z \to X$ be a holomorphic submersion of a complex space $Z$ onto a reduced Stein space $X$. Assume that $X$ is exhausted by a sequence of open subsets $U_1 \subset U_2 \subset \cdots \cup_j U_j = X$
 such that each restriction $Z\vert_{U_j} \to U_j$ is a stratified holomorphic fibre bundle whose fibres are Oka manifolds.  Then sections $X\to Z$ satisfy the following
 
{\bf Parametric Oka property (POP)}: Given a compact $\O (X)$-convex subset $K$ of $X$, a closed complex subvariety A of X, compact sets $P_0 \subset P$ in a Euclidean space $\R^m$, and a continuous map $f: P \times X \to Z$ such that

\begin{enumerate}

\item[(a)] for every $p\in P$, $f(p, \cdot) : X \to Z$ is a section of $Z\to X$ that is holomorphic on a neighborhood of $K$ (independent of $p$) and such that $f(p, \cdot)\vert_A$ is holomorphic on $A$, and

\item[(b)] $f(p, \cdot)$ is holomorphic on $X$ for every $p \in P_0$
\end{enumerate} 
 
there is a homotopy $f_t: P\times X \to Z \quad (t\in [0,1]),$ with $f_0 = f$, such that $f_t$ enjoys properties $(a)$ and $(b)$ for all  $t\in [0,1]$, and also the following hold:

\begin{enumerate}

\item[(i)] $f_1 (p, \cdot)$ is holomorphic on $X$ for all $p\in P$

\item[(ii)] $f_t$ is uniformly close to $f$ on $P\times K$ for all $t\in [0,1]$

\item[(iii)] $f_t = f$ on $(P_0 \times X) \cup (P\times A)$ for all $t\in [0,1]$

\end{enumerate}

 \end{theorem}
 
 As a general reference for Oka theory we refer to the monograph \cite{For} and the overview article \cite{Forstneric-Larusson}. 
 
 \subsection{Applications and examples} The number of applications of the Oka theory is growing, we already indicated the classical Cousin problems and Grauert's classification of holomorphic vector bundles
 over Stein spaces in the introduction. The only application  we would  like to mention is a recent solution to a problem posed by Gromov, called the Vaserstein problem. It is a natural question 
 about the $K_1$-group of the ring of holomorphic functions or in simple terms it is asking whether (and when) in a matrix whose entries are holomorphic functions (parameters) the Gauss elimination process can be performed in a way holomorphically depending on the parameter. This is to our knowledge the only application where a stratified version of an Oka theorem is needed, i.e., no proof using a non-stratified version is known.
 
 \begin{theorem}
Let $X$ be a finite dimensional reduced Stein space and $f\colon X\to \mbox{SL}_m(\mathbb{C})$ be a holomorphic mapping that is null-homotopic. Then there exist a natural number $K$ and holomorphic mappings $G_1,\dots, G_{K}\colon X\to \mathbb{C}^{m(m-1)/2}$ such that $f$ can be written as a product of upper and lower diagonal unipotent matrices
\begin{equation*}f(x) = \left(\begin{matrix} 1 & 0 \cr G_1(x) & 1 \cr \end{matrix} \right)   
\left(\begin{matrix} 1 & G_2(x) \cr 0 & 1 \cr \end{matrix} \right)  \ldots \left(\begin{matrix} 1 & G_K(x)\cr 0 & 1 \cr \end{matrix} \right) 
\end{equation*} for every $x\in X$.
\end{theorem}

Here the assumption null-homotopic means that the map is homotopic through continuous maps to a constant map (matrix), which since  Grauert's Oka principle, Theorem \ref{GOP}, is equivalent of being null-homotopic through holomorphic maps. This is an obvious necessary condition since multiplying all lower/upper diagonal
matrices in the product by $t \in [0,1]$ yields a homotopy to the (constant) identity matrix. It is a result of Vaserstein, that null-homotopic is also sufficient in order to factorize the map as a product with continuous entries. Thus we have the Oka principle. For the existence of a  holomorphic factorization there are  only topological obstructions, it exists iff a topological factorization exists.

%
%

Now we come to examples of Oka-Forstneri\v c manifolds:

 A Riemann surface is an Oka-Forstneri\v c manifold iff it is non-hyperbolic, i.e., one of $\PP^1$ $\C$, $\C^*$, or a compact torus $\C/\Gamma$. 
 
 Oka-Forstneri\v c manifolds enjoy  the following functorial properties, for elliptic manifolds (see Definition below) these properties are unknown.
 \begin{itemize}
 
 \item If $\pi: E \to B$ is a holomorphic covering map of complex manifolds then $B$ is Oka-Forstneri\v c iff $E$ is (\cite{For} Prop 5.5.2.). 
 
 \item If $E$ and $X$ are complex manifolds and $\pi: E \to X$ is a holomorphic fibre bundle whose fibre is an Oka-Forstneri\v c manifold, the $X$ is an Oka-Forstneri\v c manifold iff $E$ is (\cite{For} Theorem 5.5.4.).
 
 \item If a complex manifold $Y$ is exhausted by open domains $D_1 \subset D_2 \subset \cdots \subset \cup_{j=1}^\infty = Y$ such that every $D_j$ is an Oka-Forstneri\v c manifold, then $Y$ is an Oka-Forstneri\v c manifold. In particular
 every long $\C^n$ is an Oka-Forstneri\v c manifold. (A manifold is called a long $\C^n$ if all $D_j$ are biholomorphic to $\C^n$. If the inclusion $D_i \subset D_{i+1}$ is not a Runge pair, on those manifolds the ring of holomorphic functions may consist of constants only!!)

\end{itemize} 

The main source of examples are the elliptic manifolds (see Definition \ref{ell} below),  a notion invented by Gromov. This includes by our scheme (1) of implications all holomorphic flexible manifolds and all manifolds with the density property, in particular complex Lie groups and homogeneous  spaces
of Lie groups, the manifolds from the classical theorems of Oka and Grauert. For a Stein manifold ellipticity is equivalent to being an Oka-Forstneri\v c manifold. For general manifolds this is an open question. One possible counterexample is the complement of the ball in $\C^n$, the set
$\{ z \in \C^n  :  \vert z_1\vert^2 + \vert z_2\vert^2 + \ldots + \vert z_n\vert^2 > 1\}$. It was recently shown by Andrist and Wold \cite{AW} that it is not elliptic for $n\ge 3$, whereas it has two "nearby" properties implied by being an  Oka-Forstneri\v c manifold, strongly dominable and $\rm CAP_{n-1}$ (\cite{ForRi}).

\section{Proof of the implications from scheme (1), ellipticity in the sense of Gromov}
First remark that the two bottom up arrows in scheme (1) are obvious from the definitions. In order to prove the left right arrows let's define
 the notions introduced by Gromov \cite{Gromov} revolutionizing Oka theory (see also \cite{For} Chapter 5): 

\begin{definition}  \label{ell} Let $Y$ be a complex manifold.

\begin{enumerate} 

\item A holomorphic spray on $Y$ is a triple $(E, \pi, s)$ consisting of a holomorphic vector bundle $\pi: E \to Y$ (a spray bundle) and a holomorphic map $s: E \to Y$ (a spray map) such that for each $y \in Y$ we have $s (0_y) = y$. 

\item A spray  $(E, \pi, s)$ on $Y$ is dominating on a subset $U \subset Y$  if the differential ${\rm d}_{0_y} s  : T_{0_y} E \to T_y Y $ maps the vertical tangent space $E_y$ of $T_{0_y} E$ surjectively
onto $T_y Y$ for every $y \in U$, $s$ is dominating if this holds for all $y \in Y$.

\item A complex manifold $Y$  is elliptic if it admits a dominating holomorphic spray.

\end{enumerate}

\end{definition}

The main result of Gromov can now be formulated in the following way.

\begin{theorem}

An elliptic manifold is an Oka-Forstneri\v c manifold.

\end{theorem}
Of course Gromov proved the full Oka principle for elliptic manifolds. This proof can now be decomposed in two stages. The main
 (and the only) use of ellipticity is to prove a homotopy version of Runge (Oka-Weil) theorem, which in particular
 gives CAP  (= Oka-Forstneri\v c) and the second stage is CAP implies Oka principle.

Gromov's theorem proves our implication $${\rm\bf holomorphically  \  flexible} \Longrightarrow {\rm \bf Oka-Forstneric  \  manifold}$$ using the following example of a spray given by Gromov and Lemma \ref{finitely};

\begin{example}
Given completely integrable  holomorphic vector fields $\theta_1, \theta_2, \ldots , \theta_N$  on a complex manifold $X$ such that at each point $x \in X$ they span the tangent space,
${\rm span} (\theta_1 (x), \theta_2 (x), \ldots , \theta_N (x) = T_x X$. Let  $\psi^i : \C \times X \to X, \quad (t, x)  \mapsto \psi^i_ t (x)$ denote the corresponding one-parameter subgroups; Then
the map $s: \C^N \times X \to X$ defined by $ ((t_1, t_2, \ldots, t_n), x) \mapsto \psi_{t_n}^n \circ \psi_{t_{n-1}}^{n-1} \circ \cdots \circ \psi_{t_1}^1 (x)$ is of full rank at $t =0$ for any $x$. It is therefore a dominating spray map from the trivial bundle $X \times \C^N \to X$.  
\end{example}
\begin{lemma} \label{finitely}
If a Stein manifold $X$ is holomorphically flexible, then there are finitely many completely integrable holomorphic fields which span the tangent space $T_xX$ at every point $x \in X$
\end{lemma}

\begin{proof}To prove that there are finitely many completely integrable holomorphic fields that span each tangent space,
let us start with $n$ fields $\theta_1, \ldots, \theta_n$ which
span the tangent space at some point $x_0$ and thus outside a proper analytic subset $A$. The set $A$
may have countably many irreducible components $A_1, A_2, A_3, \ldots$.

It suffices now
to find a  holomorphic automorphism $\Phi \in \Aut_{\rm hol} (X)$ such that
$\Phi (X \setminus  A) \cap A_i \ne \emptyset$ for every  $ i = 1, 2, 3, \ldots$. Indeed,
for such an automorphism $\Phi$ the completely integrable holomorphic vector  fields $\Phi_*(\theta_1), \ldots,
\Phi_* (\theta_n)$ span the tangent space at a general point in each $A_i$, i.e., together with the fields
$\theta_1, \ldots, \theta_n$ they span the tangent space at each point outside an analytic subset $B$ of a
smaller dimension than $A$. Then the induction by dimension implies the desired conclusion.

In order to construct $\Phi$ consider a monotonically increasing sequence of compacts $K_1 \subset K_2 \subset \ldots $ in $X$ such
that $\bigcup_i K_i =X$ and a closed imbedding $\iota : X \hookrightarrow \C^m$. For every continuous map $\varphi : X \to \C^m$ denote
by $||\varphi ||_i$ the standard norm of the restriction of $\varphi$ to $K_i$.
Let $d$ be the metric on the space $ \Aut_{\rm hol}  (X)$ of holomorphic automorphisms  of $X$
given by the formula $$d(\Phi , \Psi ) =\sum_{i=1}^\infty 2^{-i}( \min (||\Phi -\Psi ||_{i},1)+ \min (||\Phi^{-1} -\Psi^{-1} ||_{i},1) \eqno{(4.1)}$$
where automorphisms $\Phi^{\pm 1}, \Psi^{\pm 1} \in \Aut_{\rm hol} (X)$  are viewed as continuous  maps from $X$ to $\C^m$.
This metric makes $ \Aut_{\rm hol}  (X)$ a complete metric space. 

Set $Z_i = \{ \Psi \in  \Aut_{\rm hol} (X) : \Psi (A_i) \cap (X\setminus A) \ne \emptyset \} $. Note that $Z_i$ is open
in $ \Aut_{\rm hol} (X)$ and let us show that it  is also everywhere dense. 

Since completely integrable holomorphic fields generates the tangent space
at each point of $X$, we can choose such a field  $\theta$  non-tangent to $A_i$. Then for every $\Psi \in  \Aut_{\rm hol} (X)$  its composition with general
elements of the flow induced by $\theta$ is  in $Z_i$.
That is,  a perturbation of $\Psi$ belongs to $Z_i$ which proves that $Z_i$ is everywhere dense in $ \Aut_{\rm hol} (X)$.
By the Baire category theorem the set $\bigcap_{i=1}^\infty Z_i$ is not empty which yields
the existence of the desired automorphism.
\end{proof}

Since the question whether holomorphic maps are approximable by morphisms is an important issue in algebraic geometry, we would like to remark  at this point that there is an application of Oka-theory to  this question.
Clearly there is an obvious notion of algebraic spray, thus algebraic ellipticity. Also the proof of the above lemma generalizes showing that an algebraically flexible manifold is algebraically elliptic.
These algebraically elliptic manifolds satisfy an algebraic version of CAP. However, in the algebraic category,
 simple examples show that algebraic CAP does not imply
 the full algebraic Oka principle, but only a weaker
 statement that being approximable by algebraic morphisms
 is a homotopy invariant property (at least for maps
 from affine algebraic manifolds to algebraically elliptic
 manifolds). For a precise treatment of this question we refer to \cite{For}  7.10.
\\
\\

The implication ${\rm  \bf DP} \Longrightarrow {\rm\bf holomorphically \  flexible } $ is contained in the following Lemma

\begin{lemma}\label{span} If a Stein manifold $X$ has the density property, the completely integrable holomorphic vector fields span the tangent space at each point $x \in X$.

\end{lemma}

\begin{proof}
It follows from  the   density property
that  Lie combinations of completely integrable holomorphic vector fields  span the tangent space $T_xX$  at any given
point $x \in X$. Observe that every Lie bracket $[\nu , \mu ]$ of completely integrable holomorphic  vector fields 
can be approximated by a linear combination of such fields which follows immediately from
the equality $[ \nu , \mu ] = \lim_{t \to 0}  {\frac{\phi_t^* (\nu ) - \nu}{t}}$ where $\phi_t$ is
the flow generated by $\mu$. Thus the completely integrable holomorphic vector fields span  $T_x X$ at any $x \in X$.
\end{proof}

\section {Concluding remarks, Open problems}

There is  also another property which has similar consequences as the density property for holomorphic automorphisms preserving a
volume form.

\begin{definition}\label{1.20}
Let a complex manifold  $X$ be
equipped with a holomorphic volume form $\omega$
(i.e. $\omega$ is nowhere vanishing section of the canonical
bundle). We say that $X$ has the volume density property (VDP) with
respect to $\omega$ if in the compact-open topology the Lie
algebra  generated by completely integrable  holomorphic
vector fields $\nu$ such that $\nu (\omega)=0$, is dense in the
Lie algebra of all holomorphic vector fields that
annihilate $\omega$ (note that condition $\nu (\omega)=0$ is
equivalent to the fact that $\nu$ is of $\omega$-divergence zero).
If $X$ is affine algebraic we say that $X$ has the algebraic
volume density property (AVDP) with respect to an algebraic volume form $\omega$ if the Lie
algebra  generated by completely integrable algebraic vector
fields $\nu$ such that $\nu (\omega)=0$, coincides with the Lie
algebra  of all algebraic vector fields that annihilate
$\omega$.

\end{definition}
 
For Stein manifolds with the volume density property (VDP) an Anders\'en-Lempert theorem for volume preserving maps holds. The implication {\bf (AVDP)} $\Rightarrow$ {\bf (VDP)} holds but its proof is not
trivial (see \cite{KaKu4}). Also {\bf (VDP)} $\Rightarrow$ {\bf holomorphic flexibility} is true (see \cite{KaKu3}). Thus we can have a scheme of implications like (1) with (DP) replaced by  (VDP) and (ADP)
replaced by (AVDP).

Volume density property and density property are not implied by each other, if $X$ has density property it may not even admit a holomorphic volume form, if $X$ has  volume density property
with respect to one volume form it may not have it with respect to another volume form and there is no reason to expect it has density property. For example, $(\C^*)^n$ for $\quad n>1$ has (algebraic) volume density property with respect to the Haar form, it does not have algebraic density property \cite{A1} and it is expected not to have density property. It is a potential counterexample to the reverse of the left horizontal arrow in scheme (1). 

Concerning the reverse implications in scheme (1):  The variety $(\C^*)^n, n>1 $  is an obvious counterexample to the reverse of the right vertical  arrow, the others are more delicate.

\smallskip\noindent
\textbf{Open Problem:} 
Which of the other three implications in scheme (1) are  reversible for a Stein manifold (resp. smooth affine algebraic variety for the vertical arrow)?
\smallskip\noindent

The main problem here  is that  no method is known how to classify (meaning exclude the existence of any other than the obvious) completely integrable holomorphic vector fields on Stein manifolds with any of our
flexibility properties. There is not even a classification of completely integrable holomorphic vector fields on $\C^2$ available.


\begin{thebibliography}{99}



\bibitem{A} E.~Anders\'en, \emph{Volume-preserving automorphisms of $\C^n$}, Complex Variables Theory Appl. \textbf{14} (1990), no. 1-4, 223--235.

\bibitem{A1} E.~Anders\'en, \emph{Complete vector fields on $(\C^*)^n$}, Proc. Amer. Math.
Soc. \textbf{128} (2000), no. 4, 1079--1085.

\bibitem{AL} E.~Anders\'en, L.~Lempert, \emph{On the group of holomorphic
automorphisms of $\C^n$}, Invent. Math.  \textbf{110} (1992), no.
2, 371--388.

\bibitem{AFKKZ} I.~ Arzhantsev, H.~ Flenner, S.~ Kaliman, F.~ Kutzschebauch, M.~ Zaidenberg,
{\em Flexible varieties and automorphism groups},   Duke Math. J. \textbf{162}, (2013), no.4, 767--823






\bibitem{AW} R.~ Andrist, E.~ Fornaess Wold, {\em The complement of the closed unit ball in $\mathbb C^3$ is not subelliptic.} arXiv:1303.1804 (2013)



  
\bibitem{D}  F. ~Donzelli {\em Algebraic density property of Danilov-Gizatullin surfaces.}  Math. Z. \textbf{272}  (2012), no. 3-4, 1187Ð1194.

\bibitem{DDK} F. ~Donzelli, A. Dvorsky, S. Kaliman, {\em Algebraic density property
of homogeneous spaces}, Transformation Groups,
\textbf{15:3} (2010) 551-576.



\bibitem{DK}
H.~Derksen, F.~Kutzschebauch,
{\em Global holomorphic linearization of actions of compact Lie groups on
$\C^n$},
Contemporary Mathematics {\bf 222} (1999), p.~201--210.


\bibitem{DK1}
H.~Derksen, F.~Kutzschebauch, \emph{Nonlinearizable holomorphic group actions.}
 Math. Ann. \textbf{311} (1998), no. 1, 41--53.


\bibitem{EG}
Y.\ Eliashberg and M.\ Gromov, \emph{Embeddings of Stein manifolds of dimension $n$ into the affine space of dimension $3n/2+1$}, Ann.\ of Math. (2), {\bf 136} (1992), no.\ 1, 123--135.

\bibitem{FKZ} H.~ Flenner, S.~Kaliman, M.~ Zaidenberg
{\em The Gromov-Winkelmann theorem for flexible varieties.}, arXiv:1305.6417,  (2013)

\bibitem{Fo}
O.~Forster, {\em Plongements des vari\'et\'es de Stein},
Comm.\ Math.\ Helv., {\bf 45} (1970), p.~170--184.

\bibitem{For1} F.~Forstneri\v c {\em Oka manifolds.} C. R. Acad. Sci. Paris, Ser. I \textbf{347}, 1017-1020 (2009)

\bibitem{For}
Forstneri\v c, F.  \textit{Stein manifolds and holomorphic mappings.}  Ergebnisse der Mathematik und ihrer Grenzgebiete.  3.\ Folge, 56.  Springer-Verlag, 2011.

\bibitem{Forstneric-Larusson}
Forstneri\v c, F.\ and F.\ L\'arusson.  \textit{Survey of Oka theory.}  New York J.\ Math.\ \textbf{17a} (2011) 11--38.

\bibitem{ForRi} F.~Forstneri \v c, T.~Ritter, {\em Oka properties of ball complements.} arXiv:1303.2239 (2013)

\bibitem{FR} F. ~Forstneri\v c, J.-P.~Rosay,  \emph{Approximation of biholomorphic mappings by automorphisms of $\C^n$}, Invent. Math. \textbf{112} (1993), no. 2, 323--349.

\bibitem{FW}   F.~  Forstneric, E.~ Fornaess Wold   {\em Embeddings of infinitely connected planar domains into $C^2$},
Analysis and  PDE, to \textbf{6} (2013), no.2, 499-514.

\bibitem{FW1}  F.~  Forstneric, E.~ Fornaess Wold {\em Bordered Riemann surfaces in $\C^2$.} J. Math. Pures Appl. (9) 91 (2009), no. 1, 100Ð114.

\bibitem{Grauert}
Grauert, H.  \textit{Analytische Faserungen \"uber holomorph-vollst\"andigen R\"aumen.}
 Math.\ Ann. \textbf{135} (1958) 263--273. 

\bibitem{GK} H.~Grauert, H.~Kerner, {\em Approximation von holomorphen Schnittfl\"achen in Faserb\"undeln mit homogener Faser.} Arch. Math., \textbf{14}, 328--333 (1963)

\bibitem{Gromov}
Gromov, M.  \textit{Oka's principle for holomorphic sections of elliptic bundles.}  J.\ Amer.\ Math.\ Soc.\ \textbf{2} (1989) 851--897.

\bibitem{HK} Heinzner, P.; Kutzschebauch, F.: {\it An equivariant version of Grauert's Oka principle.}
 Invent. math. 119, 317-346 (1995)

\bibitem{Huckleberry}
Huckleberry, A.\ T.  \textit{Actions of groups of holomorphic transformations.}  Several complex variables, VI, 143--196,
Encyclopaedia Math.\ Sci., 69.  Springer-Verlag, 1990.

\bibitem{IK} Ivarsson, B., Kutzschebauch, F. {\it   Holomorphic factorization of  mappings into ${SL}_n(\C)$},  Ann. Math. 175 (2012), no. 1., 45--69

\bibitem{Ka}
S.~Kaliman, \emph{Extensions of isomorphisms between affine algebraic subvarieties of $k^n$ to automorphisms of $k^n$},
Proc. Amer. Math. Soc. \textbf{113} (1991), no. 2, 325--33

\bibitem{KaKu1} S.~Kaliman, F.~Kutzschebauch {\em Criteria for the density
property of complex manifolds}, Invent. Math. \textbf{172} (2008), no. 1, 71--87.

\bibitem{KaKu2} S.~Kaliman, F.~Kutzschebauch, {\em Density
property for hypersurfaces $uv=p({\bar x})$}, Math. Z. \textbf{258} (2008), no. 1, 115--131.

\bibitem{KaKu3} S.~Kaliman, F.~Kutzschebauch, {\em On the present state of the Andersen-Lempert theory}, In: Affine Algebraic Geometry: The Russell Festschrift,
85--122. Centre de Recherches Math\'ematiques. CRM Proceedings and
Lecture Notes 54, 2011.

\bibitem{KaKu4} S.~Kaliman, F.~Kutzschebauch,  {\em Algebraic volume density property of affine algebraic manifolds},
Invent. Math. \textbf{181:3} (2010) 605-647.

\bibitem{Ku}
Kutzschebauch, F.  \textit{Compact and reductive subgroups of the group of holomorphic automorphisms of $\C^n$.}  Singularities and complex analytic geometry (Kyoto, 1997).  S\= urikaisekikenky\= usho K\= oky\= uroku no.\ 1033 (1998) 81--93. 


\bibitem{KLS}  Kutzschebauch, F., Larusson, F., Schwarz G. W. {\it An Oka principle for equivariant isomorphisms} J. reine angew. Math., 22p. DOI 10.1515/crelle-2013-0064
 


\bibitem{Kutzschebauch-Lodin}
Kutzschebauch, F.\ and S.\ Lodin.  \textit{Holomorphic families of non-equivalent embeddings and of holomorphic group actions on affine space.}  Duke Math.\ J.\ \textbf{162} (2013) 49--94.


\bibitem{KLL} F.~Kutzschebauch, A.~Liendo, M.~ Leuenberger, {\em On the algebraic density property for affine toric varieties.}, preprint 2013

\bibitem{La} F.~L\' arusson,  {\em Model structures and the Oka principle.} J. Pure Appl. Algebra \textbf{192} (2004), no. 1-3, 203Ð-223.

\bibitem{R} K.~J.~Ramspott  {\em Stetige und holomorphe Schnitte in B\"undeln mit homogener Faser.} Math. Z. \textbf{89}, 234--246 (1965)


\bibitem{Sch}
J.~Sch\"urmann,
{\em Embeddings of Stein spaces into affine spaces of minimal dimension},
Math. Ann. {\bf 307} (1997), no. 3, p.~381--399.

\bibitem{Schwarz}
Schwarz, G.\ W.  \textit{Exotic algebraic group actions.}  C.\ R.\ Acad.\ Sci.\ Paris S\'er.\ I Math.\ \textbf{309} (1989) 89--94. 


\bibitem{Su}  M.~Suzuki, \emph{Sur les opérations holomorphes du groupe additif complexe
sur l'espace de deux variables complexes}, Ann. Sci. Scuola Norm. Sup. \textbf{10} (1977), no. 4, 517--546.



\bibitem{V1}  D. ~Varolin, \emph{The density property for complex manifolds
and geometric structures}, J. Geom. Anal. \textbf{11} (2001), no. 1,
135--160.

\bibitem{V2}  D. ~Varolin, \emph{The density property for complex manifolds
and geometric structures. II}, Internat. J. Math. \textbf{11}
(2000), no. 6, 837--847.

\bibitem{W}
 E.~ Forn\ae ss Wold,  \emph{Embedding Riemann surfaces properly into $\Bbb C^2$}, Internat. J. Math. \textbf{17} (2006), no. 8, 963--974.




\end{thebibliography}
\end{document}